\documentclass[12pt]{amsart}
\usepackage[utf8]{inputenc}
\usepackage{amsfonts}
\usepackage{amssymb}
\usepackage{amsmath}
\usepackage[english]{babel}
\usepackage{amsthm}
\theoremstyle{plain}
\newtheorem{theorem}{Theorem}[section]
\newtheorem{corollary}[theorem]{Corollary}
\newtheorem{lemma}[theorem]{Lemma}

\theoremstyle{definition}
\newtheorem{definition}[theorem]{Definition}
\newtheorem{remark}[theorem]{Remark}
\usepackage{geometry}
\numberwithin{equation}{section} 
\usepackage{tikz}
\usepackage{caption}
\usepackage{graphicx}
\usepackage[titletoc,toc,title]{appendix}
\usepackage{lipsum}

\newcommand\blfootnote[1]{%
  \begingroup
  \renewcommand\thefootnote{}\footnote{#1}%
  \addtocounter{footnote}{-1}%
  \endgroup
}

\usepackage[colorlinks, citecolor=blue,pagebackref,hypertexnames=false]{hyperref}

\newcounter{comcount}

\begin{document}
\title[Isoperimetric Inequality on Grushin space]{Isoperimetric Inequality for degenerate elliptic operators of Grushin type} 
\date{}
\author{Dangyang He}  
\address{Department of Mathematics, Sun Yat-sen University, Guangzhou 510275, Guangdong, P.R. China}
\email{hedy28@mail.sysu.edu.cn}

\begin{abstract}
Let \(n,m\ge 1\), \(\alpha\in(0,1)\), and \(\beta\ge 0\). For the Grushin-type operator
\[
L=-\nabla_x\!\cdot\!\bigl(|x|^{2\alpha}\nabla_x\bigr)+|x|^{2\beta}\Delta_y
\qquad \text{on } \mathbb R^n\times \mathbb R^m,
\]
we prove the isoperimetric inequality on the associated Grushin space. Equivalently, if
\[
Q=\frac{n+m(\beta+1-\alpha)}{1-\alpha},
\]
then
\[
|\Omega|^{\frac{Q-1}{Q}}\le C\,P(\Omega)
\]
for every smooth bounded domain \(\Omega\subset \mathbb R^{n+m}\).
\end{abstract}

\maketitle

\tableofcontents

\blfootnote{$\textit{2020 Mathematics Subject Classification.}$ Primary 35J70; Secondary 53C17.}

\blfootnote{$\textit{Keywords and Phrases.}$ Isoperimetric inequality, Sobolev inequality, degenerate operator.}

\section{Introduction}

Let \(n,m\ge 1\), \(\alpha\in(0,1)\), and \(\beta\ge 0\). Consider the degenerate elliptic operator
\begin{equation}\label{eq:L}
L=-\nabla_x\!\cdot\!\bigl(|x|^{2\alpha}\nabla_x\bigr)+|x|^{2\beta}\Delta_y
\qquad \text{on } \mathbb R^n\times \mathbb R^m,
\end{equation}
where \(\Delta_y\) denotes the standard Laplacian in the \(y\)-variables. We refer to \cite{RS} for the precise definition and basic properties of \(L\). Associated with \(L\) is the first-order differential operator
\[
\nabla_L:=\bigl(|x|^\alpha \nabla_x,\; |x|^\beta \nabla_y\bigr),
\]
and the corresponding control distance \(d\). We write \(B(\xi,r)\) for the \(d\)-ball centered at \(\xi\) with radius \(r\), and \(P(\Omega)\) for the perimeter of a smooth set \(\Omega\subset \mathbb R^{n+m}\) with respect to this geometry.


The purpose of this paper is to establish the isoperimetric inequality on Grushin spaces. More precisely, we prove that for every smooth bounded domain \(\Omega\subset \mathbb{R}^{n+m}\),
\begin{equation}\label{ISO_Q}
    |\Omega|^{\frac{\mathcal{Q}-1}{\mathcal{Q}}} \le C_1 P(\Omega),
\end{equation}
where \(|\cdot|\) denotes the Lebesgue measure on \(\mathbb{R}^{n+m}\), \(P(\Omega)\) denotes the perimeter of \(\Omega\), and
\begin{equation}\label{Q}
    \mathcal{Q} = \frac{n+m(\beta+1-\alpha)}{1-\alpha}
\end{equation}
is the homogeneous dimension associated with \(L\). By the co-area formula, \eqref{ISO_Q} is equivalent to the Sobolev inequality
\begin{equation}\label{S_p}\tag{$\mathrm{S}_{\mathcal{Q}}^p$}
\left(\int_{\mathbb{R}^{n+m}} |f(\xi)|^{\frac{\mathcal{Q}p}{\mathcal{Q}-p}} \, d\xi \right)^{\frac{\mathcal{Q}-p}{\mathcal{Q}}}
\le
C_p \int_{\mathbb{R}^{n+m}} |\nabla_L f(\xi)|^p \, d\xi,
\qquad \forall f\in C_c^\infty(\mathbb{R}^{n+m}),
\end{equation}
in the case \(p=1\).






We now state the main result of this paper.

\begin{theorem}\label{thm_ISO}
Suppose that \(n,m\ge 1\), \(\alpha\in(0,1)\), and \(\beta\ge 0\). Then the isoperimetric inequality \eqref{ISO_Q}, or equivalently the Sobolev inequality \eqref{S_p} with \(p=1\), holds on \(\mathbb{R}^{n+m}\).
\end{theorem}

Since \(\nabla_L\) is local, the family \eqref{S_p} enjoys a self-improvement property:
\[
(\mathrm{S}_{\mathcal Q}^p)\implies (\mathrm{S}_{\mathcal Q}^q),
\qquad 1\le p\le q<\mathcal Q.
\]
As a consequence, we obtain the following corollary.

\begin{corollary}\label{thm_S_p}
Suppose that \(n,m\ge 1\), \(\alpha\in(0,1)\), and \(\beta\ge 0\). Then the Sobolev inequality \((\mathrm{S}_{\mathcal Q}^p)\) holds on \(\mathbb{R}^{n+m}\) for all \(1\le p<\mathcal Q\).
\end{corollary}

In the Grushin plane, the isoperimetric inequality was established in \cite{MM}. In higher dimensions, Franceschi and Monti \cite{FM} proved it for \(n=1\), \(\alpha=0\), and \(m\ge 1\). Moreover, for \(n\ge 2\), they obtained the corresponding result for \(x\)-spherically symmetric sets. Finally, under the same assumptions as in Theorem~\ref{thm_ISO}, Hauer and Sikora \cite[Theorem 6.3]{HS} established the Sobolev inequality \((\mathrm{S}_{\mathcal Q}^p)\) for \(2\le p<\mathcal Q\).


Our argument is motivated by the following observation. Endow \(\mathbb{R}^{n+m}\) with the metric
\[
g=|x|^{-2\alpha}dx^2+|x|^{-2\beta}dy^2.
\]
Passing to polar coordinates in the \(x\)-variable and introducing $\rho=(1-\alpha)^{-1}r^{1-\alpha}$, one sees that \(g\) can be written as the doubly warped product metric
\[
g=d\rho^2+[(1-\alpha)\rho]^2\,d\theta^2+[(1-\alpha)\rho]^{-2\beta/(1-\alpha)}\,dy^2
\]
on
\[
(0,\infty)_\rho\times_{(1-\alpha)\rho}\mathbb{S}^{n-1}\times_{[(1-\alpha)\rho]^{-\beta/(1-\alpha)}}\mathbb{R}^m.
\]
By the standard curvature formula for doubly warped products,
\[
\textrm{Ric}_\xi \ge -\frac{C(n,m,\alpha,\beta)}{\rho^2}\,g_\xi
= -\frac{C(n,m,\alpha,\beta)}{|x|^{2-2\alpha}}\,g_\xi.
\]
Although our Grushin operator \(L\) is not the standard Laplace--Beltrami operator associated with \(g\), this computation suggests a corresponding curvature-dimension inequality for \(-L\). Combined with a remote ball argument from \cite{Gilles}, this yields a suitable gradient estimate for the heat semigroup. We then apply the method of \cite{He_ISO}, based on a weak-type Sobolev inequality and a Hardy-type inequality, to obtain the desired isoperimetric estimate.

\section{Preliminaries}

Throughout the paper, we write \(\xi=(x,y)\) and \(\eta=(x',y')\), where \(x,x'\in\mathbb R^n\) and \(y,y'\in\mathbb R^m\). We denote by \(B(\xi,r)\) the Grushin ball $B(\xi,r):=\{\eta\in\mathbb R^{n+m}: d(\xi,\eta)<r\}$, and by $V(\xi,r):=|B(\xi,r)|$ its volume with respect to the Lebesgue measure on \(\mathbb R^{n+m}\). We also write \(B_n\) and \(B_m\) for Euclidean balls in \(\mathbb R^n\) and \(\mathbb R^m\), respectively.

The following lemma summarizes some basic geometric properties of Grushin spaces; see \cite[Proposition~5.1, Theorem~6.4]{RS}.

\begin{lemma}\label{le_geo}
Let $n, m\ge 1$, $\alpha \in (0,1)$ and $\beta \ge 0$.

\begin{enumerate}
\item Distance estimate: there exists constant $C_d>0$ such that \begin{align}
   C_d^{-1} D(\xi,\eta) \le  d(\xi,\eta) \le C_d D(\xi,\eta),
\end{align}
where
\begin{equation*}
    D(\xi,\eta) = \frac{|x-x'|}{\left( |x| + |x'| \right)^\alpha} + \frac{|y-y'|}{(|x|+|x'|)^{\beta} + |y-y'|^{\frac{\beta}{\beta+1-\alpha}}}.
\end{equation*}
\item Volume estimate: \begin{align*}
    V(\xi,r) = V((x,y),r) \sim \begin{cases}
        r^{\mathcal{Q}}, & r\ge |x|^{1-\alpha},\\
        r^{n+m}|x|^{n \alpha + m\beta}, & r\le |x|^{1-\alpha}.
    \end{cases}
\end{align*}
In addition, there exists constant $C_D>0$ such that the volume doubling condition holds:
\begin{align*}
\frac{V(\xi, R)}{V(\xi,r)} \le C_D \left(\frac{R}{r}\right)^{\mathcal{Q}}
\end{align*}
for all $\xi \in \mathbb{R}^{n+m}$ and all $R\ge r>0$.
\item Gaussian upper bound for heat kernel: 
\begin{equation*}
    e^{-tL}(\xi,\eta) \le \frac{C}{V(\xi, \sqrt{t})} e^{-\frac{d(\xi,\eta)^2}{ct}}
\end{equation*}
for all $t>0$ and almost every $\xi,\eta \in \mathbb{R}^{n+m}$.
\end{enumerate}

\end{lemma}

With Lemma~\ref{le_geo}, we develop the following property.

\begin{lemma}\label{le_volume}
Let $n,m\ge 1$, $\alpha\in (0,1)$ and $\beta\ge 0$. Let $\xi=(x,y)\in \mathbb{R}^{n+m}$. There exist $0<c_1<c_2$ and $\epsilon_0 \in (0,1)$ such that for all $r\le \epsilon_0 |x|^{1-\alpha}$, one has
\begin{equation}
    B_n(x, c_1 r |x|^\alpha) \times B_m(y, c_1 r |x|^\beta) \subset B(\xi, r) \subset B_n(x, c_2 r |x|^\alpha) \times B_m(y, c_2 r |x|^\beta).
\end{equation}
Moreover, there exist $0<c_3<c_4$ such that for all $r>0$, one has for all $y\in \mathbb{R}^m$
\begin{equation}\label{le_volume3}
    B_n(0, c_3 r) \times B_m(y, c_3 r^{\beta+1-\alpha}) \subset B((0,y), r^{1-\alpha}) \subset B_n(0, c_4 r) \times B_m(y, c_4 r^{\beta+1-\alpha}).
\end{equation}
\end{lemma}

\begin{proof}[Proof of Lemma~\ref{le_volume}]
We first show that 
\begin{equation}\label{le_volume1}
    B(\xi, r) \subset B_n(x, c_2 r |x|^\alpha) \times B_m(y, c_2 r |x|^\beta).
\end{equation}
Take $\eta = (x',y')\in B(\xi, r)$. Then by Lemma~\ref{le_geo}, $D(\xi,\eta)\le C_d r$. We claim that
\begin{equation}\label{claim1}
    |x-x'| \le \frac{|x|}{2},
\end{equation}
provided $\epsilon_0$ is small enough. Indeed, if not, then $|x|+|x'|\le |x-x'| + 2|x| \le 3|x-x'|$ and thus
\begin{equation*}
    \frac{|x-x'|}{\left(|x|+|x'|\right)^\alpha} \ge 3^{-\alpha} |x-x'|^{1-\alpha} \ge 3^{-\alpha} 2^{\alpha-1} |x|^{1-\alpha},
\end{equation*}
which further implies that $D(\xi,\eta)\ge 3^{-\alpha} 2^{\alpha-1} |x|^{1-\alpha}$. Pick $\epsilon_0 < C_d^{-1} 3^{-\alpha} 2^{\alpha-1}$. Then the previous argument contradicts the fact that $D(\xi,\eta)\le C_d r$ and hence the claim \eqref{claim1} is verified. With \eqref{claim1}, one gets
\begin{equation}\label{eq_volume2}
    \frac{|x|}{2} \le |x'| \le \frac{3}{2}|x|\quad \textrm{and}\quad |x|+|x'| \le \frac{5}{2} |x|.
\end{equation}

Next, since $D(\xi,\eta)\le C_d r$, we must have
\begin{align*}
    \frac{|x-x'|}{\left(|x|+|x'|\right)^\alpha} \le C_d r \implies |x-x'|\le C_d r \left(|x|+|x'|\right)^\alpha \le C_d \left(\frac{5}{2}\right)^\alpha r |x|^\alpha.
\end{align*}
While for the $y$-part of $D(\xi,\eta)$, if $\left(|x|+|x'|\right)^{\beta+1-\alpha}\ge |y-y'|$ then one obtains
\begin{align*}
\frac{|y-y'|}{2\left(|x|+|x'|\right)^{\beta}} \le \frac{|y-y'|}{\left(|x|+|x'|\right)^{\beta} + |y-y'|^{\frac{\beta}{\beta+1-\alpha}}} \le C_d r.
\end{align*}
By \eqref{eq_volume2}, one yields
\begin{equation*}
    |y-y'|\le 2C_d r \left(|x|+|x'|\right)^{\beta} \le 2C_d \left(\frac{5}{2}\right)^\beta r |x|^\beta.
\end{equation*}
If $\left(|x|+|x'|\right)^{\beta+1-\alpha}\le |y-y'|$, then
\begin{align*}
|y-y'|^{\frac{1-\alpha}{\beta+1-\alpha}} \le 2 C_d r \implies |y-y'|\le \left(2C_d\right)^{\frac{\beta+1-\alpha}{1-\alpha}} r r^{\frac{\beta}{1-\alpha}}\le \left(2C_d\right)^{\frac{\beta+1-\alpha}{1-\alpha}} r |x|^{\beta}.
\end{align*}
Setting $c_2 = \max \left( C_d \left(\frac{5}{2}\right)^\alpha,  2C_d \left(\frac{5}{2}\right)^\beta, \left(2C_d\right)^{\frac{\beta+1-\alpha}{1-\alpha}} \right)$, one concludes \eqref{le_volume1}.

The proof of the converse direction:
\begin{equation}
    B_n(x, c_1 r |x|^\alpha) \times B_m(y, c_1 r |x|^\beta) \subset B(\xi, r)
\end{equation}
is obvious. Indeed, for $ \eta = (x',y')$ such that $|x-x'|\le c_1 r |x|^\alpha$ and $|y-y'|\le c_1 r |x|^\beta$, one directly bounds
\begin{equation*}
    d(\xi,\eta)\le C_d D(\xi,\eta)\le C_d \frac{c_1 r |x|^\alpha}{\left(|x|+|x'|\right)^\alpha} + C_d \frac{c_1 r |x|^\beta}{\left(|x|+|x'|\right)^\beta+|y-y'|^{\frac{\beta}{\beta+1-\alpha}}} \le 2 C_d c_1 r.
\end{equation*}
Choosing $c_1 = \frac{1}{4C_d}$, one confirms $\eta \in B(\xi,r)$ as desired.

The proof for \eqref{le_volume3} is similar and we omit it.
\end{proof}

\begin{remark}
Note that the constants \(c_1\) and \(c_2\) depend only on \(C_d\), \(\alpha\), and \(\beta\). Therefore, by choosing $\epsilon_0< \min \left( C_d^{-1} 3^{-\alpha} 2^{\alpha-1}, \frac{1}{2c_2} \right)$, one sees that for every \(r\le \epsilon_0 |x|^{1-\alpha}\) and every $\eta=(x',y')\in B_n(x, c_2 r |x|^\alpha)\times B_m(y, c_2 r |x|^\beta)$, we have
\[
|x|\le |x-x'|+|x'|
\le c_2 r |x|^\alpha + |x'|
\le c_2\epsilon_0 |x|+|x'|
\le \frac{|x|}{2}+|x'|.
\]
Hence $|x|\le 2|x'|\le 3|x|$.
\end{remark}

\begin{definition}
A ball \(B=B(\xi,r)=B((x,y),r)\) is said to be \emph{remote} if
\[
r\le \epsilon_0 |x|^{1-\alpha},
\]
where \(\epsilon_0\) is given by Lemma~\ref{le_volume}. For such a ball, every \(\eta=(x',y')\in B\) satisfies
\[
|x-x'|\le \frac{|x|}{2}
\qquad\text{and hence}\qquad
|x|\sim |x'|.
\]
\end{definition}

Next, we recall the following result from \cite[Theorem~3.1, Proposition~3.5]{BCLS}.

\begin{lemma}\label{le_BCLS}
Let $0<p<\infty$, $r_0,s_0\in (0,\infty]$ and $\theta_0 \in (0,1]$ and that the parameter $q = q(r_0,s_0,\theta_0)$ defined by
\begin{align*}
    \frac{1}{r_0}  = \frac{\theta_0}{q} + \frac{1-\theta_0}{s_0},
\end{align*}
satisfies $p\le q<\infty$. Suppose that
\begin{equation}
    \sup_{\lambda>0} \lambda  \left| \{|f|\ge \lambda\}\right|^{\frac{1}{r_0}}  \le C \|\nabla_L f\|_p^{\theta_0} \|f\|_\infty^{1-\theta_0} |\textrm{supp}(f)|^{\frac{1-\theta_0}{s_0}}.
\end{equation}
Then, we have
\begin{align*}
    \|f\|_q \le C \| \nabla_L f\|_p,\quad q = \frac{r_0 s_0 \theta_0}{s_0 - r_0(1-\theta_0)}
\end{align*}
for all $f\in C_c^\infty(\mathbb{R}^{n+m})$.
\end{lemma}

\begin{remark}
In particular, if we take $p=1$ and show that
\begin{equation}\label{showthis}
\sup_{\lambda>0} \lambda^{\frac{\mathcal{Q}}{\mathcal{Q}-1}}  \left|\left\{ |f| \ge \lambda \right\}\right| \le C \|\nabla_L f\|_1 \|f\|_q^{\frac{1}{\mathcal{Q}-1}},
\end{equation}
then Lemma~\ref{le_BCLS} yields $(\textrm{S}_{\mathcal{Q}}^1)$ and hence \eqref{ISO_Q}.
\end{remark}

\section{Proof of Theorem~\ref{thm_ISO}}

Given a symmetric Markov generator $\mathcal{L}$ with carré du champ $\Gamma$ and iterate $\Gamma_2$, the inequality
\begin{align*}
    \Gamma_2(f) \ge K \Gamma(f) + \frac{1}{N} \left(L(f)\right)^2,\quad \forall f\in C^\infty
\end{align*}
encodes both a 'Ricci curvature' lower bound $K$ and 'effective dimension' $N$. For conventional issues, we let $\mathcal{L}$ be a \textbf{non-positive} self-adjoint diffusion operator. Formally, for smooth function $f,g$ one defines bilinear form:
\begin{align}\label{carre1}
    \Gamma(f,g) = \frac{1}{2} \left( \mathcal{L}(fg) - f \mathcal{L}(g) - g\mathcal{L}(f)\right).
\end{align}
Note that $\Gamma$ is symmetric and $\Gamma(f) = \Gamma(f,f)$ is the so-called \textit{carre du champ} operator. Next, one defines the iterated bilinear form:
\begin{align}\label{carre2}
    \Gamma_2(f,g) = \frac{1}{2} \left( \mathcal{L}\Gamma(f,g) - \Gamma(f,\mathcal{L}(g)) - \Gamma(g,\mathcal{L}(f)) \right),
\end{align}
and $\Gamma_2(f) = \Gamma_2(f,f)$ is the second \textit{carre du champ} operator.

\begin{lemma}\label{le_Ric_Grushin}
Let $n,m\ge 1$, $\alpha \in (0,1)$ and $\beta \ge 0$. Let $\Gamma$ (resp. $\Gamma_2$) be the \textit{carre du champ} (resp. second \textit{carre du champ}) of $L$. The following curvature dimension inequality holds
\begin{align*}
\Gamma_2(f) \ge - \frac{C(n,m,\alpha,\beta)}{|x|^{2(1-\alpha)}} \Gamma(f) + C'(n,m,\alpha,\beta) \left(L(f) \right)^2
\end{align*}
for all $f\in C^\infty( \mathbb{R}^{n}\setminus\{0\} \times \mathbb{R}^m)$.
\end{lemma}

\begin{proof}[Proof of Lemma~\ref{le_Ric_Grushin}]
Since $L$ is non-negative, we set $\mathcal{L} = -L$ and calculate \eqref{carre1} and \eqref{carre2} directly. Put $r=|x|$, $a = r^{2\alpha}$ and $b=r^{2\beta}$. Then
\begin{equation*}
\mathcal L = -a\Delta_x - b\Delta_y+\nabla a\cdot \nabla_x.
\end{equation*}

For simplicity, we use indices $i,j\in\{1,\dots,n\}$, $\mu,\nu\in\{1,\dots,m\}$, one gets directly from the definition \eqref{carre1} that
\begin{equation}\label{eq_gamma1}
    \Gamma(f)=a f_i^2+b f_\mu^2.
\end{equation}

Next, by \eqref{carre2}, a direct calculation gives
\begin{align}\label{eq_gamma2}
\Gamma_2(f)
={}&a^2 f_{ij}^2+2ab f_{i\mu}^2+b^2 f_{\mu\nu}^2\\ \nonumber
&+2a a_i f_j f_{ij} + a a_j f_j \Delta_x f\\ \nonumber
&+2a b_i f_\mu f_{i\mu}+ a b_i f_i \Delta_y f\\ \nonumber
&+\Bigl(-\frac a2 \Delta_x a+\frac12|\nabla a|^2\Bigr)f_i^2
-a a_{ij}f_i f_j\\ \nonumber
&+\Bigl(-\frac a2 \Delta_x b+\frac12\nabla a\cdot \nabla b\Bigr)f_\mu^2.
\end{align}
Note that
\begin{align*}
    a_i=2\alpha r^{2\alpha-2}x_i,\quad
|\nabla a|^2 = 4\alpha^2 r^{4\alpha-2},
\quad
\Delta_x a = -2\alpha(n+2\alpha-2)r^{2\alpha-2},
\end{align*}
\begin{equation*}
    a_{ij} = 2\alpha r^{2\alpha-2}\delta_{ij}
 + 4\alpha(\alpha-1)r^{2\alpha-4}x_i x_j,
\end{equation*}
and 
\begin{align*}
    b_i = 2\beta r^{2\beta-2}x_i,\quad
|\nabla b|^2 = 4\beta^2 r^{4\beta-2},
\quad
\Delta_x b = -2\beta(n+2\beta-2)r^{2\beta-2}.
\end{align*}

We estimate the lower-order terms in \eqref{eq_gamma2} by Young's inequality.

First, using $|\Delta_x f|\le \sqrt n|D_x^2 f|$, where $|D_x^2f|^2 = \sum_{1\le i,j\le n}|\partial_{x_i,x_j}f|^2$, and Cauchy-Schwartz inequality and then Young's inequality, one obtains
\begin{equation}\label{1}
    2a a_i f_j f_{ij} + a a_j f_j \Delta_x f \ge
-\frac14 a^2|D_x^2 f|^2-C|\nabla a|^2|\nabla_x f|^2.
\end{equation}
Also, since $\alpha<1$,
\begin{equation}\label{2}
-a a_{ij}f_i f_j = -2\alpha r^{4\alpha-2}|\nabla_x f|^2 + 4\alpha(1-\alpha)r^{4\alpha-4}(x\cdot \nabla_x f)^2
\ge -2\alpha r^{4\alpha-2}|\nabla_x f|^2.
\end{equation}
Hence
\begin{equation}\label{3}
    \Bigl(-\frac a2 \Delta_x a + \frac12| \nabla a|^2\Bigr)f_i^2-a a_{ij}f_i f_j \ge
-C r^{4\alpha-2}|\nabla_x f|^2.
\end{equation}

Next, it is clear that
\begin{equation}\label{4}
    2a b_i f_\mu f_{i\mu}
\ge
-a b f_{i\mu}^2-\frac{a|\nabla b|^2}{b} |\nabla_y f|^2,
\end{equation}
and, using $|\Delta_y f|\le \sqrt m|D_y^2 f|$, where $|D_y^2f|^2 = \sum_{1\le \mu,\nu\le m}|\partial_{y_\mu,y_\nu}f|^2$, 
\begin{equation}\label{5}
    a b_i f_i \Delta_y f \ge -\frac14 b^2|D_y^2 f|^2 - C \frac{a^2|\nabla b|^2}{b^2}|\nabla_x f|^2.
\end{equation}

Finally,
\begin{equation}\label{6}
    \Bigl( -\frac a2 \Delta_x b + \frac12 \nabla a \cdot \nabla b \Bigr) f_\mu^2 \ge
 -  C r^{2\alpha+2\beta-2}|\nabla_y f|^2.
\end{equation}

Insert \eqref{1}--\eqref{6} into \eqref{eq_gamma2}. 
We obtain
\begin{align}\label{7}
{}&\Gamma_2(f)\\ \nonumber
&\ge \frac34 a^2|D_x^2 f|^2 + ab|D_{xy}^2 f|^2 + \frac34 b^2|D_y^2 f|^2 - C \Bigl(r^{4\alpha-2}|\nabla_x f|^2 + r^{2\alpha+2\beta-2}|\nabla_y f|^2\Bigr)\\ \nonumber
&= \frac34 a^2|D_x^2 f|^2 + ab|D_{xy}^2 f|^2 + \frac34 b^2|D_y^2 f|^2 -\frac{C}{|x|^{2-2\alpha}}\Gamma(f)
\end{align}

On the other hand, by trace inequality, with $C'=12(n+m)$, 
\begin{align}\label{8}
C'^{-1}(\mathcal L f)^2 &\le
C'^{-1} 3a^2(\Delta_x f)^2+ C'^{-1} 3b^2(\Delta_y f)^2+ C'^{-1}3|\nabla a|^2|\nabla_x f|^2\\ \nonumber
&\le C'^{-1} 3n a^2|D_x^2 f|^2 + C'^{-1} 3m b^2|D_y^2 f|^2 + C'^{-1} 12\alpha^2 r^{4\alpha-2}|\nabla_x f|^2\\ \nonumber
&\le \frac14 a^2|D_x^2 f|^2+\frac14 b^2|D_y^2 f|^2+\frac1{|x|^{2-2\alpha}}\Gamma(f).
\end{align}

Combining \eqref{7} and \eqref{8}, we conclude that
\begin{equation*}
    \Gamma_2(f)\ge -\frac{C}{|x|^{2-2\alpha}}\Gamma(f)+\frac1{C'}(\mathcal L f)^2
\end{equation*}
for some $C=C(n,m,\alpha,\beta)>0$, as desired.

\end{proof}

Next, we establish a gradient estimate for the heat semigroup. The argument is inspired by \cite[Sections~3.2--3.3]{Gilles}.

\begin{lemma}\label{le_Grushin_gradient}
Let $n, m\ge 1$, $\alpha \in (0,1)$ and $\beta \ge 0$. The following gradient estimate holds
\begin{align*}
    \left|\nabla_{L}e^{-tL}(\xi,\eta)\right|\le \left(\frac{1}{\sqrt{t}} + \frac{1}{|x|^{1-\alpha}}\right) \frac{C}{V(\xi,\sqrt{t})}e^{-\frac{d(\xi,\eta)^2}{ct}},
\end{align*}
where $\xi = (x,y) \in \mathbb{R}^n \setminus \{0\} \times \mathbb{R}^m$.
\end{lemma}



\begin{proof}[Proof of Lemma~\ref{le_Grushin_gradient}]
Let \(\xi=(x,y)\in \mathbb{R}^n\setminus\{0\}\times \mathbb{R}^m\), and let \(B=B(\xi,r)\) be a remote ball with radius \(r\le \epsilon_0 |x|^{1-\alpha}\). By Lemma~\ref{le_Ric_Grushin}, one readily checks that for every \(\xi'\in B\), the curvature-dimension inequality \(\mathrm{CD}(-c_1/r^2,c_2)\) holds uniformly for some constants \(c_1,c_2>0\). Therefore, by \cite[Theorem~1.4]{ZZ}; see also \cite[Theorem~5.1]{XDL}, a local Li--Yau type estimate implies that for every \(\xi'\in B/2\), every \(\gamma>1\), and every positive solution \(u_t\) of
\[
\partial_t u_t=-Lu_t,
\]
one has
\begin{equation}\label{le_grushin_grad1}
    \frac{|\nabla_L u_t|^2}{u_t^2}-\gamma \frac{\partial_t u_t}{u_t}
    \le
    C(\gamma,n,m)\left(\frac{1}{t}+\frac{1}{r^2}\right).
\end{equation}
Combining \eqref{le_grushin_grad1} with \cite[Theorem~2.6]{Sturm3}, namely
\begin{equation}\label{le_grushin_grad2}
    \left| \partial_t e^{-tL}(\xi,\eta)\right|
    \le
    \frac{C}{t\,V(\xi,\sqrt t)}e^{-\frac{d(\xi,\eta)^2}{ct}},
\end{equation}
yields the desired gradient estimate.
\end{proof}

The next step is to establish the following Hardy type inequality.
\begin{lemma}\label{le_hardy}
Let $n, m\ge 1$, $\alpha \in (0,1)$ and $\beta \ge 0$. Then, the following Hardy type inequality holds
\begin{align}\label{eq_hardy}
    \int_{\mathbb{R}^{n+m}} \frac{|u(\xi)|}{|x|^{1-\alpha}} d\xi \le C_H \int_{\mathbb{R}^{n+m}} |\nabla_{L}u(\xi)| d\xi,
\end{align}
for all $u\in C_c^\infty(\mathbb{R}^{n+m})$.
\end{lemma}

\begin{proof}[Proof of Lemma~\ref{le_hardy}]
We give a direct proof. Assume first that \(n\ge 2\). Passing to polar coordinates in the \(x\)-variable, the left-hand side of \eqref{eq_hardy} becomes
\begin{align*}
    C(n) \int_{\mathbb{R}^m} \int_0^\infty \int_{\mathbb{S}^{n-1}}
    \frac{|u((r,\theta);y)|}{r^{1-\alpha}}\, d\theta\, r^{n-1}\, dr\, dy.
\end{align*}
Using the inequality
\begin{align*}
    |u((r,\theta);y)|
    \le
    \int_r^\infty |\partial_s u((s,\theta);y)|\, ds,
\end{align*}
we deduce that the left-hand side of \eqref{eq_hardy} is bounded by
\begin{align*}
    C(n) \int_{\mathbb{R}^m} \int_{\mathbb{S}^{n-1}} \int_0^\infty \int_r^\infty
    |\partial_s u((s,\theta);y)|\, ds\, r^{n-2+\alpha}\, dr\, d\theta\, dy.
\end{align*}
Since \(r\le s\) in the region of integration, we have \(r^\alpha\le s^\alpha\). Moreover,
\[
|s^\alpha \partial_s u|\le |\nabla_L u|.
\]
Therefore, the above expression is bounded by
\begin{align*}
    C(n) \int_{\mathbb{R}^m} \int_{\mathbb{S}^{n-1}} \int_0^\infty \int_0^s
    |\nabla_L u((s,\theta);y)|\, r^{n-2}\, dr\, ds\, d\theta\, dy
    \\
    \le
    C(n) \int_{\mathbb{R}^m} \int_{\mathbb{S}^{n-1}} \int_0^\infty
    |\nabla_L u((s,\theta);y)|\, s^{n-1}\, ds\, d\theta\, dy
    \\
    \le
    C(n) \int_{\mathbb{R}^{n+m}} |\nabla_L u(\xi)|\, d\xi.
\end{align*}
The case \(n=1\) is similar, and we omit the details.
\end{proof}

The next lemma is a variant of the argument in \cite[Section~2.1]{DR_hardy}. We include a proof for completeness.

\begin{lemma}\label{le_remoteball}
Let $R>0$. There exists a sequence of balls $\{B_n^j = B(x_j, r_j)\}_{j \ge 0}$ such that 
\begin{enumerate}
    \item $\mathbb{R}^n = \cup_{j\ge 0} B_n^j$, where $B_n^0 = B_n(0,R)$ and $B_j$ ($j\ge 1$) is $x$-remote in the sense $r_j \le |x_j|/2$,
    \item for all $j \ge 1$, $2^{-10} |x_j| \le r_j \le 2^{-9} |x_j|$,
    \item $\sum_j \mathbf{1}_{B_j}(x) \le C_L$ for all $x\in \mathbb{R}^n$, where $C_L>0$ does not depend on $R$.
\end{enumerate}

\end{lemma}

\begin{proof}[Proof of Lemma~\ref{le_remoteball}]
Set $B_n^0 = B_n(0,R)$ and $A_N:= B_n(0, R 2^{N}) \setminus B_n(0, R 2^{N-1})$ for each $N\ge 1$. Apparently, 
\begin{align*}
    \mathbb{R}^n = B_n^0 \cup \bigcup_{N\ge 1}A_N,\quad A_N \subset \bigcup_{x\in A_N} B_n\big(x, R 2^{N-13}\big).
\end{align*}
It follows by Vitali's covering lemma, that one can find an index set $I_N$ and a collection of balls $\{B_n\big(x_{N,j},R2^{N-13}\big)\}_{j\in I_N}$ with $x_{N,j}\in A_N$, pairwise disjoint and $A_N \subset \bigcup_{j\in I_N}B_n\big(x_{N,j}, R2^{N-10}\big)$. One deduces the finiteness of $\# I_N$, 
\begin{align*}
    \# I_N |B_n(0,R 2^{N+2})| &\le \sum_{j \in I_N} |B_n\big(x_{N,j}, R2^{N+3}\big)|\\
    &\le  2^{16n} \sum_{j \in I_N} |B_n(x_{N,j}, R2^{N-13})| \\
    &\le 2^{16n}|B_n \big(0,R2^{N+2}\big)|,
\end{align*}
Note that by setting $B_n^{j} = B_n(x_{j}, r_j)$ with $x_j = x_{N,j}$, $r_j = R 2^{N-10}$ and relabeling, we construct a sequence of balls $\{B_n^j\}_j$ such that $\cup_{j \ge 0}B_n^j = \mathbb{R}^n$. Moreover, we have for $j \ne 0$ (since $x_j \in A_N$ for some $N$),
\begin{align*}
2^{-10}|x_j|\le r_j \le 2^{-9} |x_j|,
\end{align*}
and for all $x\in B_n^j$,
\begin{align*}
    2^8 r_j \le |x| \le 2^{11}r_j.
\end{align*}
Clearly, this construction guarantees that $B_n^0 = B_n(0,R)$ is the only ball in the collection which contains $0$. In addition, for $x\ne 0$, one sets $J_x = \{j; x\in B_n^j\}$. Observe that if $x\in A_N$ for some $N\ge 1$ (the case $x\in B_n^0$ is similar), and $x\in B_n^j$ for some $j$, then $j$ is either in $I_N$, $I_{N-1}$ or $I_{N+1}$. Therefore, if $x\in A_N$, we have
\begin{align*}
    \# J_x &\big|B_n\big(x,2^{-1}|x| \big)\big| \le \sum_{j \in J_x} \big|B_n \big(x_j, 2^{-1}|x|+2^{-8}|x|\big)\big| \le 2^{20n} \sum_{j\in J_x} \big|B_n \big(x_j, 2^{-20}|x|\big) \big|\\
    &\le 2^{20n} \sum_{j\in J_x} \big|B_n\big(x_j, 2^{-3}r_j\big)\big| = 2^{20n} \sum_{l\in \{-1,0,1\}} \sum_{j \in J_x \cap I_{N+l}} \big|B_n\big(x_j, 2^{-3}r_j\big)\big| \\
    &=2^{20n} \sum_{l\in \{-1,0,1\}} \left|\cup_{j \in J_x\cap I_{N+l}}B_n\big(x_j, 2^{-3} r_j\big)\right|\\
    &\le 3 \cdot 2^{20n} \big|B_n\big(x, 2^{-1}|x|\big)\big|,
\end{align*}
since for all $j \in J_x\cap I_{N+l}$, one has $$B_n\big(x_j, 2^{-3} r_j\big) \subset B_n\big(x,9 \cdot 2^{-3} r_j\big) \subset B_n \big(x, 9\cdot 2^{-11} |x| \big) \subset B_n\big(x, 2^{-1}|x|\big).$$ 
The case $x\in B_n^0$ is similar. Hence, we conclude that there exists a constant $C_L>0$ which does not depend on $R,N$ such that
\begin{align*}
    \sum_j \mathbf{1}_{B_j}(x) \le C_L,\quad \forall x\in \mathbb{R}^n,
\end{align*}
completing the proof.
\end{proof}

We are now in a position to prove Theorem~\ref{thm_ISO}

\begin{proof}[Proof of Theorem~\ref{thm_ISO}]

Let $\lambda>0$, $q=\frac{\mathcal{Q}}{\mathcal{Q}-1}$ and $f\in C_c^\infty(\mathbb{R}^{n+m})$ be fixed. Set $R = \left(\frac{\|f\|_q}{\lambda}\right)^{\frac{q}{\mathcal{Q}(1-\alpha)}}$. By Lemma~\ref{le_remoteball} with parameter $R$, there exists a sequence of balls $\{B_n^{j}=B_n(x_j,r_j)\}_{j \ge 0}$ such that
\begin{align*}
    \mathbb{R}^n = B_n^0 \cup \left(\cup_{j \ge 1}B_n^{j} \right).
\end{align*}
Moreover, each $B_n^{j}$ ($j \ne0$) is $x$-remote and $B_n^0 = B_n(0,R)$. Let $\{\mathcal{X}_\alpha\}$ be a smooth partition of unity subordinated to $B_n^{j}$ such that $0\le \mathcal{X}_j \le 1$ and
\begin{enumerate}
    \item $\sum_{j\ge 1} \mathcal{X}_{j} + \mathcal{X}_{0} = 1$,
    \item $\|\mathcal{X}_j\|_\infty +  r_j \|\nabla_x \mathcal{X}_j\|_\infty \le C$ for all $j\ge 1$,
    \item $\textrm{supp}(\mathcal{X}_j) \subset B_n^{j}$ for all $j\ge 0$.
\end{enumerate}
Next, we decompose
\begin{equation*}
    f(x,y) = \sum_{j\ge 1} f(x,y) \mathcal{X}_j(x) + f(x,y)\mathcal{X}_{0}(x):= \sum_{j \ge 1} f_j + f_{0}.
\end{equation*}
Since the covering has finite overlap property (and the finite overlap constant does not depend on $R$), one infers
\begin{align}\label{eqlast}
    \left| \{\mathbb{R}^{n+m}; |f|\ge \lambda\} \right| \le \sum_{j\ge 0} \mathcal{F}_j,
\end{align}
where
\begin{equation*}
    \mathcal{F}_j = \left| \{B_n^{j} \times \mathbb{R}^m; |f_j|\ge C_L^{-1}\lambda\} \right|,\quad \forall j\ge 0.
\end{equation*}

Now, for each $j \ge 1$, we further split (for some $t>0$)
\begin{equation*}
    f_j = \left( f_j - e^{-tL}f_j \right) + e^{-tL}f_j.
\end{equation*}
Note that by \cite[Prop~3.1]{DS2}, we have $\|e^{-tL}\|_{1\to \infty}\lesssim t^{-\mathcal{Q}/2}$. It follows by interpolation that $\|e^{-tL}f_j\|_\infty \le C t^{-\frac{\mathcal{Q}}{2q}} \|f\|_q$. Therefore, by setting $t = \left(2C_L C \|f\|_q/\lambda\right)^{2q/\mathcal{Q}}$, we deduce that
$$
\left|\{|e^{-tL}f_j|\ge (2C_L)^{-1}\lambda\}\right|= 0.
$$
Thus we obtain for $j\ge 1$
\begin{align}\label{eq_thm_main1}
    \mathcal{F}_j \le \left|\left\{B_n^{j} \times \mathbb{R}^m; |f_j - e^{-tL}f_j| \ge (2C_L)^{-1}\lambda  \right\}\right|.
\end{align}

To estimate \eqref{eq_thm_main1}, we write $f_j - e^{tL}f_j = \int_0^t L e^{-sL}f_j ds$. With Minkowski's integral inequality, $\mathcal{F}_j$ is bounded by
\begin{align}\label{eq_thm_main2}
    2C_L\lambda^{-1} \int_0^t \|L e^{-sL} f_j \|_{L^1(B_n^{j} \times \mathbb{R}^m)} ds.
\end{align}
Let $g\in L^\infty$ supported in $B_n^{j} \times \mathbb{R}^m$ with $\|g\|_{\infty} \le 1$. By integration by parts and Hölder's inequality, it is clear that
\begin{align}\label{eq_thm_main3}
    \iint L e^{-sL} f_j(x,y) g(x,y) dx dy &= \int_{\mathbb{R}^m}\int_{B_n^{j}} \nabla_L f_j(x,y) \cdot \nabla_L e^{-sL}g(x,y) dx dy\\ \nonumber
    &\le \left\| \nabla_L f_j \right\|_{L^1(\mathbb{R}^{n+m})} \left\| \nabla_L e^{-tL}g\right\|_{L^\infty(B_n^{j}\times \mathbb{R}^m)}.
\end{align}

\medskip

\textit{Claim.}
\begin{align}\label{claim}
    \sup_{j \ge 1} \sup_{s>0} \|\sqrt{s} \nabla_L e^{-sL}g \|_{L^\infty(B_n^{j}\times \mathbb{R}^m)} \le  C\|g\|_\infty
\end{align}
for all $g\in L^\infty$ supported in $B_n^{j} \times \mathbb{R}^m$, where $C$ depends only on $n,m,\alpha,\beta,C_d,C_D$.

\begin{proof}[Proof of \textit{Claim.}]

Let $\xi = (x,y)\in B_n^j \times\mathbb{R}^m$, $\eta=(x',y') \in B_n^j \times\mathbb{R}^m$ and $s>0$. Note that by Lemma~\ref{le_remoteball}, one has $2^8 r_j \le |x| \le 2^{11} r_j$. Then Lemma~\ref{le_Grushin_gradient} implies that
\begin{align*}
    \left|\nabla_L e^{-sL}g(x,y)\right| &\le \int_{\mathbb{R}^m} \int_{B_n^j} \left(\frac{1}{\sqrt{s}} + \frac{1}{r_j^{1-\alpha}}\right) \frac{C}{V(\xi,\sqrt{s})} e^{-\frac{d(\xi,\eta)^2}{cs}} |g(\eta)| d\eta.
\end{align*}

\textit{Case 1.} If $\sqrt{s}\le 2^{11(1-\alpha)}r_j^{1-\alpha}$, then a standard argument (c.f. \cite[Lemma 2.1]{CD2}) yields that
\begin{align*}
\left|\nabla_L e^{-sL}g(\xi)\right| &\le s^{-1/2} \|g\|_\infty \frac{C}{V(\xi,\sqrt{s})} \int_{\mathbb{R}^{n+m}} e^{-\frac{d(\xi,\eta)^2}{cs}}d\eta \le C s^{-1/2} \|g\|_\infty,
\end{align*}
where the constant $C>0$ depends only on $C_D$.

\textit{Case 2.} If $\sqrt{s}\ge 2^{11(1-\alpha)}r_j^{1-\alpha} \ge |x|^{1-\alpha}$, one deduces
\begin{align}\label{eq_le_gradient1}
\left|\nabla_L e^{-sL}g(\xi)\right| &\le r_j^{\alpha-1} \|g\|_\infty \frac{C}{V(\xi,\sqrt{s})} \int_{B_n^j} \int_{\mathbb{R}^m} e^{-\frac{d(\xi,\eta)^2}{cs}} d y' dx'.
\end{align}
By Lemma~\ref{le_geo}, one has volume estimate $V(\xi,\sqrt{s}) = V((x,y), \sqrt{s}) \sim s^{\mathcal{Q}/2}$. Also, the distance estimate gives
\begin{align*}
    d((x,y);(x',y'))^2 \sim \frac{|x-x'|^2}{\left( |x| + |x'| \right)^{2\alpha}} + \frac{|y-y'|^2}{(|x|+|x'|)^{2\beta} + |y-y'|^{\frac{2\beta}{\beta+1-\alpha}}}.
\end{align*}
Transferring to polar coordinates system (set $\sigma = (|x|+|x'|)^{2\beta}$), the inner integral of \eqref{eq_le_gradient1} is bounded by (up to a constant depending on $C_d$)
\begin{align*}
\int_0^\infty &\textrm{exp}\left(\frac{-cr^2 s^{-1}}{\sigma + r^{2\beta/(\beta+1-\alpha)} } \right) r^{m-1} dr \\
&\le \int_0^{\sigma^{\frac{\beta+1-\alpha}{2\beta}}} \textrm{exp}\left( - \frac{r^2}{2s\sigma} \right) r^{m-1} dr + \int_{\sigma^{\frac{\beta+1-\alpha}{2\beta}}}^\infty \textrm{exp}\left( -\frac{r^{\frac{2-2\alpha}{\beta+1-\alpha}}}{2s} \right) r^{m-1} dr\\
&\le C\max \left(s^{\frac{m}{2}} \sigma^{\frac{m}{2}}, s^{\frac{m(\beta+1-\alpha)}{2(1-\alpha)}} \right) \le C s^{\frac{m(\beta+1-\alpha)}{2(1-\alpha)}},
\end{align*}
where the last inequality follows by observing that $\sigma \sim |x|^{2\beta} \sim r_j^{2\beta} \le C s^{\frac{\beta}{1-\alpha}}$. 

Combining this with \eqref{eq_le_gradient1}, one infers
\begin{align*}
\left|\nabla_L e^{-sL}g(\xi)\right| &\le C\|g\|_\infty r_j^{\alpha-1} s^{-\frac{\mathcal{Q}}{2}} |B_n^j| s^{\frac{m(\beta+1-\alpha)}{2(1-\alpha)}}\\
&\le C \|g\|_\infty r_j^{n+\alpha-1} s^{-\frac{n}{2(1-\alpha)}}\\
&\le Cs^{\frac{n-(1-\alpha)}{2(1-\alpha)} - \frac{n}{2(1-\alpha)}}  \|g\|_\infty  \\
&= C s^{-1/2}\|g\|_\infty
\end{align*}
as desired. This proves the \textit{Claim.}
\end{proof}

\medskip

By the \textit{Claim} \eqref{claim}, \eqref{eq_thm_main2} and \eqref{eq_thm_main1}, one obtains for $j\ge 1$,
\begin{align*}
    \mathcal{F}_j \le C C_L\lambda^{-1} \|\nabla_L f_j\|_1 t^{\frac{1}{2}} &= CC_L \lambda^{-1} \|\nabla_L f_j\|_1 \left(\frac{\|f\|_q}{\lambda}\right)^{\frac{q}{\mathcal{Q}}} \\
    &= CC_L \lambda^{-1-q/\mathcal{Q}} \|\nabla_L f_j\|_1 \|f\|_q^{q/\mathcal{Q}}.
\end{align*}
Note that for $x\in B_n^j$
\begin{align}
    |\nabla_L f_j| &= \left|\left( |x|^\alpha \mathcal{X}_j \nabla_xf  + |x|^\alpha f \nabla \mathcal{X}_j, |x|^\beta \mathcal{X}_j \nabla_yf \right)\right|\\ \nonumber&\le C \left( |\nabla_L f| + |x|^\alpha \frac{|f|}{r_j} \right)\\ \nonumber
    &\le C \left(|\nabla_L f| + \frac{|f|}{|x|^{1-\alpha}}\right).
\end{align}
It then follows by finite overlap property and Lemma~\ref{le_hardy} that
\begin{align*}
\sum_{j \ge 1} \mathcal{F}_j &\le CC_L \lambda^{-1-q/\mathcal{Q}} \|f\|_q^{q/\mathcal{Q}} \left( \sum_{j \ge 1} \|\nabla_L f\|_{L^1(B_n^j \times \mathbb{R}^m)} +  \sum_{j \ge 1} \left\|\frac{f}{|x|^{1-\alpha}}\right\|_{L^1(B_n^j \times \mathbb{R}^m)} \right)\\
&\le CC_L \lambda^{-1-q/\mathcal{Q}} \|f\|_q^{q/\mathcal{Q}} \left( \|\nabla_L f\|_{L^1(\mathbb{R}^{n+m})} + \int_{\mathbb{R}^m} \sum_{j \ge 1} \int_{B_n^j} \frac{|f(x,y)|}{|x|^{1-\alpha}} dx dy \right)\\
&\le CC_L(1+C_H) \lambda^{-1-q/\mathcal{Q}} \|f\|_q^{q/\mathcal{Q}}\|\nabla_L f\|_1.
\end{align*}

To this end, one also needs to treat $\mathcal{F}_0$. By Lemma~\ref{le_hardy}, it is plain that
\begin{align*}
    \mathcal{F}_0 &= \left|\{B_{0}\times \mathbb{R}^m; |f_{0}|\ge C_L^{-1}\lambda\}\right|\\
    &\le C_L \lambda^{-1} \| f \|_{L^1\left(B_n(0,R)\times \mathbb{R}^m\right)} \le C_L \lambda^{-1} R^{1-\alpha} \left\| 
\frac{f}{|x|^{1-\alpha}} \right\|_{L^1(\mathbb{R}^n \times \mathbb{R}^m)}\\
&\le C_LC_H \lambda^{-1-q/\mathcal{Q}} \|f\|_q^{q/\mathcal{Q}}\|\nabla_L f\|_1.
\end{align*}
Consequently, one concludes by \eqref{eqlast} that
\begin{align*}
    \sup_{\lambda>0} \lambda^{1+q/\mathcal{Q}} \left|\left\{ (x,y)\in \mathbb{R}^{n+m}; |f(x,y)|\ge \lambda \right\}\right| \le C \|\nabla_L f\|_1 \|f\|_q^{q/\mathcal{Q}},
\end{align*}
where the constant $C$ only depends on $n,m,\alpha,\beta,C_d,C_L,C_H,C_D$. 

The proof of Theorem~\ref{thm_ISO} now follows from Lemma~\ref{le_BCLS}.
\end{proof}

\bigskip



\bibliographystyle{abbrv}

\bibliography{references.bib}


\end{document}